\documentclass[11pt]{article}

\usepackage{amsmath,amssymb,amsfonts,amsthm,mathabx,enumitem}
\usepackage[pdftex]{graphicx}

\numberwithin{equation}{section}

\newtheorem{thm}{Theorem}[section]
\newtheorem{prp}[thm]{Proposition}
\newtheorem{lmm}[thm]{Lemma}

\newtheorem{crl}[thm]{Corollary}

\theoremstyle{definition}
\newtheorem{dfn}[thm]{Definition}

\def\BE#1{\begin{equation}\label{#1}}
\def\EE{\end{equation}}
\def\eref#1{(\ref{#1})}
\def\ov#1{\overline{#1}}
\def\tn#1{\textnormal{#1}}
\def\wt#1{\widetilde{#1}}
\def\wh#1{\widehat{#1}}

\def\sf#1{\textsf{#1}}

\def\lr#1{\langle{#1}\rangle}
\def\flr#1{\left\lfloor{#1}\right\rfloor}

\def\lra{\longrightarrow}

\def\ga{\gamma}

\def\si{\sigma}

\def\vph{\varphi}
\def\vt{\vartheta}

\def\ze{\zeta}

\def\Ga{\Gamma}
\def\La{\Lambda}

\def\Si{\Sigma}
\def\Th{\Theta}

\def\fa{\mathfrak a}
\def\C{\mathbb C}
\def\cC{\mathcal C}
\def\fc{\mathfrak c}

\def\bI{\mathbb I}
\def\fI{\mathfrak i}
\def\fj{\mathfrak j}

\def\P{\mathbb P}
\def\R{\mathbb R}

\def\Z{\mathbb Z}

\def\Aut{\textnormal{Aut}}
\def\tnd{\textnormal{d}}
\def\id{\textnormal{id}}
\def\Fr{\textnormal{Fr}}
\def\Hom{\textnormal{Hom}}
\def\Id{\textnormal{Id}}
\def\Isom{\textnormal{Isom}}
\def\GL{\textnormal{GL}}
\def\rk{\textnormal{rk}}
\def\SL{\textnormal{SL}}
\def\SU{\textnormal{SU}}
\def\top{\textnormal{top}}

\def\Spin{\textnormal{Spin}}
\def\SO{\textnormal{SO}}

\def\prt{\partial}
\def\dbar{\bar\partial}
\def\eset{\emptyset}
\def\w{\wedge}
\def\rdet{\wh{\tn{det}}}

\begin{document}

\title{On the Topology of Real Bundle Pairs\\ 
over Nodal Symmetric Surfaces}
\author{Penka Georgieva\thanks{Supported by ERC grant STEIN-259118} $~$and 
Aleksey Zinger\thanks{Partially supported by NSF grant DMS 1500875 and MPIM}}
\date{\today}
\maketitle

\begin{abstract}
\noindent
We give an alternative argument for the classification of real bundle pairs
over smooth symmetric surfaces and extend this classification to nodal symmetric surfaces.
We also classify the homotopy classes of automorphisms of  real bundle pairs
over symmetric surfaces.
The two statements together describe the isomorphisms between  real bundle pairs
over symmetric surfaces up to deformation.
\end{abstract}

\section{Introduction}
\label{intro_sec}

\noindent
The study of symmetric surfaces goes back to at least~\cite{Klein}.
They have since played important roles in different areas of mathematics,
as indicated by~\cite{AG} and its citations.
Real bundle pairs, or \sf{Real vector bundles} in the sense of~\cite{Atiyah}, 
over smooth symmetric surfaces are classified in~\cite{BHH}.
In this paper, we give an alternative proof of this core result of~\cite{BHH},
obtain its analogue for nodal symmetric surfaces, and classify
the automorphisms of real bundle pairs over symmetric surfaces.
Special cases of the main results of this paper, 
Theorems~\ref{nodalBHH_thm} and~\ref{auto_thm} below, 
are one of the ingredients in the construction of positive-genus real Gromov-Witten invariants
in~\cite{RealGWsI} and in the study of their properties in~\cite{RealGWsII}.\\

\noindent
An \textsf{involution} on a topological space~$X$ is a homeomorphism
$\phi\!:X\!\lra\!X$ such that $\phi\!\circ\!\phi\!=\!\id_X$.
A \sf{symmetric surface} $(\Si,\si)$ is a closed  oriented (possibly nodal) surface~$\Si$  
with an orientation-reversing involution~$\si$.
If $\Si$ is smooth, the fixed locus~$\Si^{\si}$ of~$\si$ is a disjoint union of circles.
In general, $\Si^{\si}$ consists of isolated points (called $E$~nodes in 
\cite[Section~3.2]{Melissa})
and circles identified at pairs of points (called $H$~nodes in~\cite[Section~3.2]{Melissa}).\\

\noindent
Let $(X,\phi)$ be a topological space with an involution.
A \sf{conjugation} on a complex vector bundle $V\!\lra\!X$ 
\sf{lifting}~$\phi$ is a vector bundle homomorphism 
$\vph\!:V\!\lra\!V$ covering~$\phi$ (or equivalently 
a vector bundle homomorphism  $\vph\!:V\!\lra\!\phi^*V$ covering~$\id_X$)
such that the restriction of~$\vph$ to each fiber is anti-complex linear
and $\vph\!\circ\!\vph\!=\!\id_V$.
A \sf{real bundle pair} $(V,\vph)\!\lra\!(X,\phi)$   
consists of a complex vector bundle $V\!\lra\!X$ and 
a conjugation~$\vph$ on $V$ lifting~$\phi$.
For example, 
$$(X\!\times\!\C^n,\phi\!\times\!\fc)\lra(X,\phi),$$
where $\fc\!:\C^n\!\lra\!\C^n$ is the standard conjugation on~$\C^n$,
is a real bundle pair;
we call it the \sf{trivial rank~$n$ real bundle pair over~$(X,\phi)$}.
For any real bundle pair $(V,\vph)$ over~$(X,\phi)$, 
the fixed locus 
$$V^{\vph}\equiv \big\{v\!\in\!V\!:\,\vph(v)\!=\!v\big\}$$
of~$\vph$
is a real vector bundle over the fixed locus $X^{\phi}$ of~$\phi$
with $\rk_{\R}V^{\vph}\!=\!\rk_{\C}V$.\\

\noindent
If $(V_1,\vph_1)$ and $(V_2,\vph_2)$ are real vector bundle pairs over~$(X,\phi)$,
an \sf{isomorphism}
\BE{nodalBHH_e}\Phi\!:(V_1,\vph_1)\lra (V_2,\vph_2)\EE
\sf{of real bundle pairs over~$(X,\phi)$} is a $\C$-linear isomorphism $\Phi\!:V_1\!\lra\!V_2$
covering the identity~$\id_X$ such that $\Phi\!\circ\!\vph_1\!=\!\vph_2\!\circ\!\Phi$.
We call two real bundle pairs  $(V_1,\vph_1)$ and $(V_2,\vph_2)$ over~$(X,\phi)$ \sf{isomorphic} if 
there exists an isomorphism of real bundle pairs as in~\eref{nodalBHH_e}.
Our first theorem classifies real bundle pairs over symmetric surfaces up to isomorphism. 

\begin{thm}\label{nodalBHH_thm}
Suppose $(\Si,\si)$ is a (possibly nodal) symmetric surface.
Two real bundle pairs $(V_1,\vph_1)$ and $(V_2,\vph_2)$ over $(\Si,\si)$ are
isomorphic if and only~if 
$$\rk_{\C}V_1=\rk_{\C}V_2, \qquad
w_1\big(V_1^{\vph_1}\big)=w_1\big(V_2^{\vph_2}\big)\in H^1(\Si^{\si};\Z_2),$$
and $\deg(V_1|_{\Si'})\!=\!\deg(V_2|_{\Si'})$ for each irreducible component $\Si'\!\subset\!\Si$.
\end{thm}

\noindent
Let $X$ be a topological space. 
We denote by~$\cC(X;\R^*)$ and~$\cC(X;\C^*)$  the topological groups of \hbox{$\R^*$-valued} 
and  $\C^*$-valued, respectively,  continuous functions on~$X$.
For a real vector bundle~$V$ over~$X$,
let $\GL(V)$ be the topological group of vector bundle isomorphisms
of~$V$ with itself covering~$\id_X$ and
$\SL(V)\!\subset\!\GL(V)$ be the subgroup of isomorphisms~$\psi$ so that the induced isomorphism
$$\La_{\R}^{\top}\psi\!:\La^{\top}_{\R}V\lra \La^{\top}_{\R}V$$
is the identity.
If $V$ is a line bundle, then $\GL(V)$ is naturally identified with~$\cC(X;\R^*)$ and
$\SL(V)\!\subset\!\GL(V)$ is the one-point set consisting of the constant function~1.
For an arbitrary real vector bundle~$V$ over~$X$ and $\psi\!\in\!\GL(V)$,
we denote by ${\det}_{\R}\psi$ the continuous function on~$X$ corresponding to
the isomorphism $\La_{\R}^{\top}\psi$ of~$\La_{\R}^{\top}V$.\\

\noindent
Let $(X,\phi)$ be a topological space with an involution.
Denote~by 
$$\cC\big(X,\phi;\C^*\big)\subset \cC\big(X;\C^*\big)$$ 
the subgroup of continuous maps~$f$
such that \hbox{$f(\phi(z))\!=\!\ov{f(z)}$} for all $z\!\in\!X$.
The restriction of such a function to the fixed locus $X^{\phi}\!\subset\!\Si$ 
takes values in~$\R^*$, i.e.~gives rise to a homomorphism
$$\cC\big(X,\phi;\C^*\big)\lra \cC\big(X;\R^*\big), \qquad f\lra f\big|_{X^{\phi}}\,.$$
For a real bundle pair~$(V,\vph)$ over~$(X,\phi)$,
let $\GL(V,\vph)$ be the topological group of real bundle isomorphisms 
of~$(V,\vph)$ with itself over~$(X,\phi)$ and 
$\SL(V,\vph)\!\subset\!\GL(V,\vph)$
be the subgroup of isomorphisms~$\Psi$ so that the induced isomorphism
$$\La_{\C}^{\top}\Psi\!:\La^{\top}_{\C}(V,\vph)\lra \La^{\top}_{\C}(V,\vph)$$
is the identity.
If $(V,\vph)$ is a rank~1 real bundle pair, $\GL(V,\vph)$ is naturally identified 
with~$\cC(X,\phi;\C^*)$ and
$\SL(V,\vph)$ is the one-point set consisting of the constant function~1.
For an arbitrary real vector bundle pair $(V,\vph)$ and $\Psi\!\in\!\GL(V,\vph)$,
we denote by ${\det}_{\C}\Psi$ the element of $\cC(X,\phi;\C^*)$  corresponding to
the isomorphism  $\La_{\C}^{\top}\Psi$ of~$\La^{\top}_{\C}(V,\vph)$.
Let
$$\GL'(V,\vph)=\big\{(f,\psi)\!\in\!\cC(X,\phi;\C^*)\!\times\!\GL(V^{\vph})\!:\,
f|_{X^{\phi}}\!=\!{\det}_{\R}\psi\big\}.$$
Our second theorem describes the topological components of $\GL(V,\vph)$
and $\SL(V,\vph)$ for real bundle pairs over symmetric surfaces.

\begin{thm}\label{auto_thm}
Let $(\Si,\si)$ be a (possibly nodal) symmetric surface and $(V,\vph)$ be a real bundle pair
over~$(\Si,\si)$.
Then the homomorphisms
\BE{AutoThm_e}\begin{split}
\GL(V,\vph)\lra \GL'(V,\vph)&\lra \cC(\Si,\si;\C^*),\GL(V^{\vph}),\\
\Psi\lra\big(\det\Psi,\Psi|_{V^{\vph}}\big), &\qquad
(f,\psi)\lra f,\psi,
\end{split}\EE
are surjective.
Two automorphisms of $(V,\vph)$ lie in the same path component
of~$\GL(V,\vph)$ if and only if their images in $\cC(\Si,\si;\C^*)$ and in $\GL(V^{\vph})$
lie in the same path components of the two spaces.
Furthermore, every path $(f_t,\psi_t)$ in $\GL'(V,\vph)$ passing through 
the images of some \hbox{$\Psi,\Phi\!\in\!\GL(V,\vph)$} lifts to a path 
in $\GL(V,\vph)$ passing through~$\Psi$ and~$\Phi$. 
The analogous statements hold for the homomorphism
\BE{SLrest_e}\SL(V,\vph)\lra \SL(V^{\vph}), \qquad \Psi\lra\Psi|_{V^{\vph}}, \EE
in place of~\eref{AutoThm_e}.
\end{thm}

\noindent
For a smooth symmetric surface $(\Si,\si)$, Theorem~\ref{nodalBHH_thm} reduces to
\cite[Propositions~4.1,4.2]{BHH}.
We give a completely different proof of this result in Section~\ref{BHH_sec};
see Proposition~\ref{BHH_prp} and its proof.
The portion of Theorem~\ref{auto_thm} concerning the surjectivity
of the homomorphism~\eref{SLrest_e} and its analogue for $\GL(V,\vph)$
is established in Section~\ref{Isom_sec}; see Proposition~\ref{IsomExtend_prp}.
We use this proposition
to complete the proof of Theorem~\ref{nodalBHH_thm} by induction from the base case 
of Proposition~\ref{BHH_prp} in Section~\ref{IsomExist_sec}.
Proposition~\ref{homotopextend_prp} is the crucial step needed for the lifting of homotopies in
Theorem~\ref{auto_thm}; it is obtained in Section~\ref{Homotopies_sec}.
This theorem is then proved in Section~\ref{MainPfs_sec}.
Section~\ref{realGW_sec} describes connections of Theorems~\ref{nodalBHH_thm} and~\ref{auto_thm}
with recent advances in real Gromov-Witten theory made in \cite{RealOrient,RealGWsI,RealGWsII}.\\

\noindent
The second author would like to thank Max-Planck-Institut f\"ur Mathematik
for the hospitality during the preparation of this paper.

\section{The smooth case of Theorem~\ref{nodalBHH_thm}}
\label{BHH_sec}

\noindent
We begin by establishing the smooth case of Theorem~\ref{nodalBHH_thm}.

\begin{prp}[{\cite[Propositions~4.1,4.2]{BHH}}]\label{BHH_prp}
Theorem~\ref{nodalBHH_thm} holds if $(\Si,\si)$ is a smooth symmetric surface.
\end{prp}

\noindent
Let $(\Si,\si)$ be a smooth genus~$g$ symmetric surface.
We denote by $|\si|_0\!\in\!\Z^{\ge0}$ the number of connected components of~$\Si^{\si}$;
each of them is a circle.
Let $\lr\si\!=\!0$ if the quotient $\Si/\si$ is orientable, 
i.e.~$ \Si\!-\! \Si^{\si}$ is disconnected, and $\lr\si\!=\!1$ otherwise. 
There are $\flr{\frac{3g+4}{2}}$ different topological types of orientation-reversing 
involutions $\si$ on $\Si$ classified by the triples $(g,|\si|_0,\lr\si)$; 
see \cite[Corollary~1.1]{Nat}. \\

\noindent
An \sf{oriented symmetric half-surface} (or simply \sf{oriented sh-surface}) 
is a pair $(\Si^b,c)$ consisting of an oriented bordered smooth surface~$\Si^b$ 
and an involution $c\!:\prt\Si^b\!\lra\!\prt\Si^b$ preserving each component
and the orientation of~$\prt\Si^b$.
The restriction of~$c$  to a boundary component 
is either the identity or the antipodal map
$$\fa\!:S^1\lra S^1, \qquad z\lra -z,$$
for a suitable identification of $(\prt\Si^b)_i$ with $S^1\!\subset\!\C$;
the latter type of boundary structure is called \sf{crosscap} 
in the string theory literature.
We define
$$c_i=c|_{(\prt\Si^b)_i}, \qquad 
|c_i|= \begin{cases} 0,&\hbox{if}~c_i=\id;\\ 1,&\hbox{otherwise};\end{cases}
\qquad
|c|_k=\big|\{(\prt\Si^b)_i\!\subset\!\Si^b\!:\,|c_i|\!=\!k\}\big|\quad k=0,1.$$
Thus, $|c|_0$ is the number of standard boundary components of $(\Si^b,\prt\Si^b)$  
and $|c|_1$ is the number of crosscaps.
Up to isomorphism, each oriented sh-surface $(\Si^b,c)$ is determined by the genus~$g$ of~$\Si^b$,
the number~$|c|_0$ of ordinary boundary components, and the number~$|c|_1$ of crosscaps.\\

\begin{figure}
\begin{center}
\leavevmode
\includegraphics[width=0.6\textwidth,height=150px]{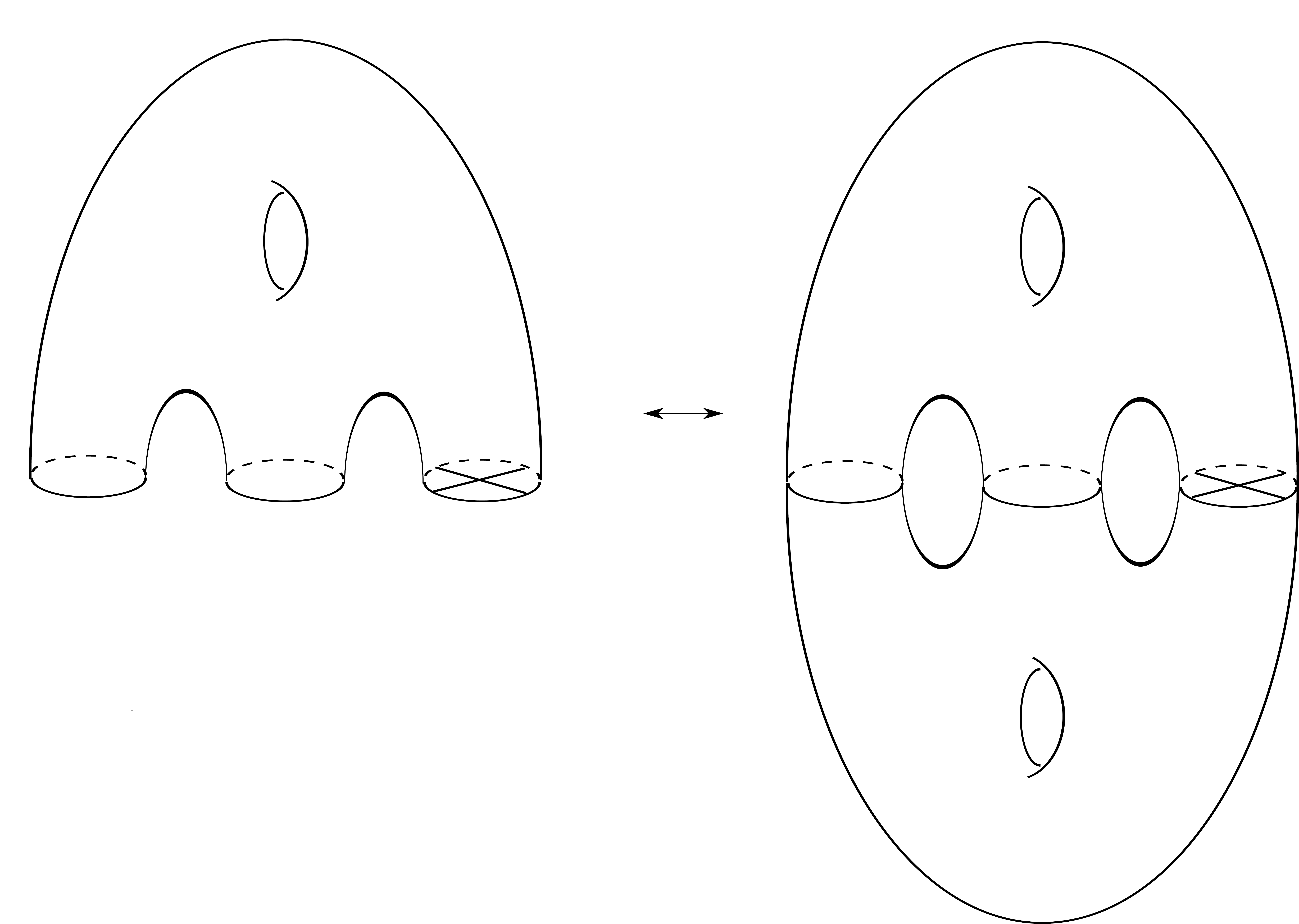}
\end{center}
\caption{Doubling an oriented sh-surface}
\label{HSdoub_fig}
\end{figure}

\noindent
An  oriented sh-surface $(\Si^b,c)$ of type $(g,m_0,m_1)$ \sf{doubles} 
to a symmetric surface~$(\Si, \si)$ of type 
$$(g(\Si), |\si|_0,\lr\si)=\begin{cases} (2g\!+\!m_0\!+\!m_1\!-\!1,m_0,0),& \text{if}~ m_1=0;\\
(2g\!+\!m_0\!+\!m_1\!-\!1,m_0,1),& \text{if}~m_1\neq 0;
\end{cases}$$
so that $\si$ restricts to~$c$ on the cutting circles (the boundary of~$\Si^b$);
see \cite[(1.6)]{XCapsSetup} and Figure~\ref{HSdoub_fig}.
Since this doubling construction covers all topological types of orientation-reversing 
involutions $\si$ on $\Si$, for every symmetric surface $(\Si,\si)$ there is 
an oriented sh-surface $(\Si^b,c)$ which doubles to~$(\Si,\si)$.
In general,  the topological type of such an sh-surface is not unique.\\

\noindent
Let $(\Si,\si)$ be a smooth symmetric surface and $(\Si^b,c)$ be an oriented sh-surface
doubling to~$(\Si,\si)$.
For each~$i$, let 
$$(\prt\Si^b)_i\!\times\!(-2,2)\lra \Si $$
be a parametrization of 
a neighborhood~$U_i$ of~$(\prt\Si^b)_i$ such that $(\prt\Si^b)_i\!\times\!0$ 
corresponds to~$(\prt\Si^b)_i$ and
$$\si(x,\tau)=(x,-\tau) \qquad\forall~(x,\tau)\!\in\!(\prt\Si^b)_i\!\times\!(-2,2).$$ 
We assume that these neighborhoods are disjoint.
Let $(\Si',\si')$ be the nodal symmetric surface obtained from~$(\Si,\si)$
by collapsing the circles $\tau\!=\!\pm1$ in each~$U_i$.
Since $(\prt\Si^b)_i$ is a separating collection, 
the surface $\Si'$ consists of two closed surfaces, $\Si_+'$ and $\Si_-'$, interchanged by~$\si'$
and attached to a collection~$\{S^2_i\}$ of $\si'$-invariant spheres 
with $(\prt\Si^b)_i\!\subset\!S^2_i$; see Figure~\ref{collapse_fig}.
We will call the latter the \sf{central components} of~$\Si'$.
Let
$$q_{\Si}\!:\Si\lra\Si'$$
be the quotient map.
In particular, $q_{\Si}\!\circ\!\si\!=\!\si'\!\circ\!q_{\Si}$.\\

\begin{figure}
\begin{center}
\leavevmode
\includegraphics[width=0.6\textwidth,height=150px]{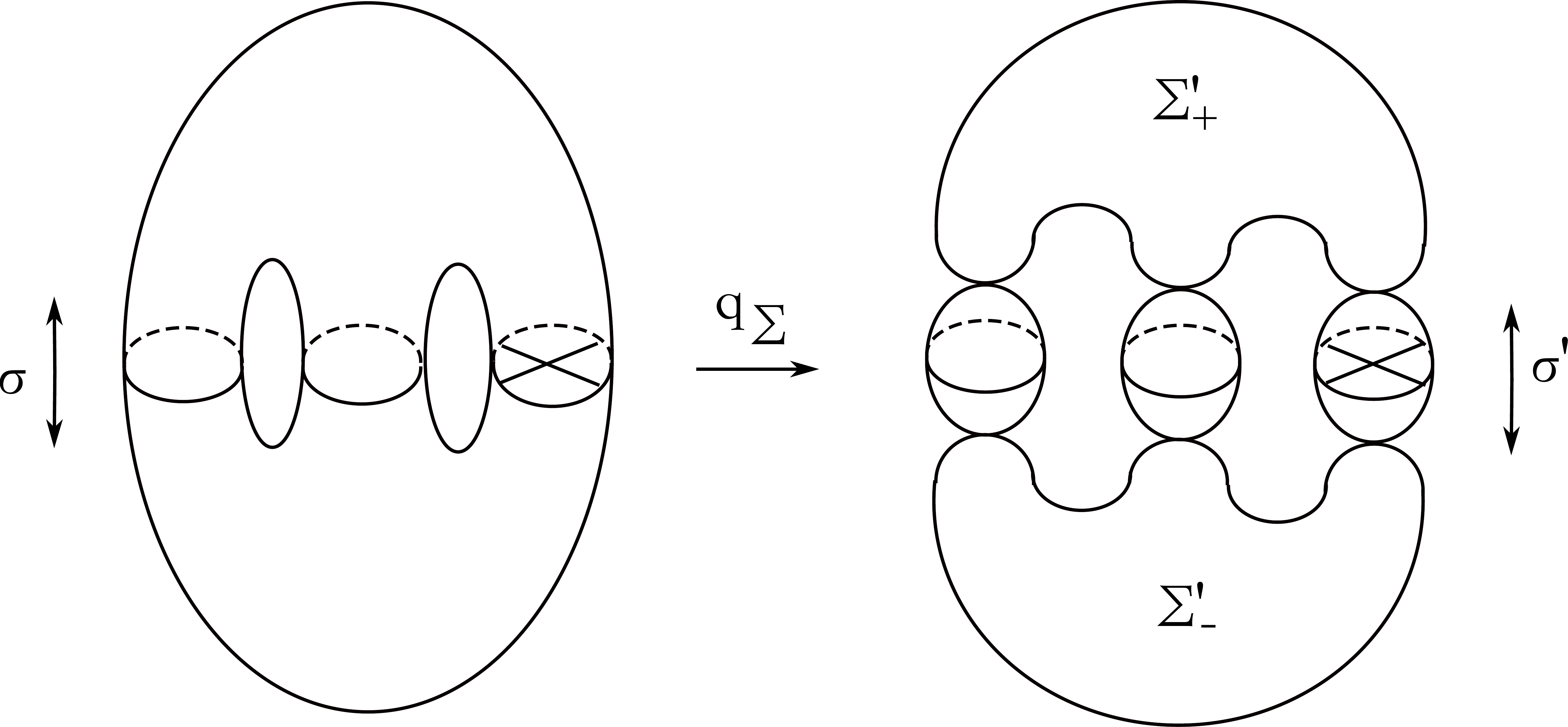}
\end{center}
\caption{A smooth symmetric surface~$(\Si,\si)$ and its associated pinched surface~$(\Si',\si')$}
\label{collapse_fig}
\end{figure}

\noindent
For each cutting circle $(\prt\Si^b)_i$ with $|c_i|\!=\!0$,  let 
$$D_i^+=q_{\Si}\big((\prt\Si^b)_i\!\times\![0,1]\big),\qquad
D_i^-=q_{\Si}\big((\prt\Si^b)_i\!\times\![-1,0]\big).$$
Choose a homeomorphism $f_i\!: (\prt\Si^b)_i\!\lra\!S^1$
and define a rank~1 real bundle pair $(\ga_i,\wt\si_i')$ over~$(S^2_i,\si')$~by
\begin{gather*}
\ga_i\equiv \big(D_i^+\!\times\!\C\sqcup D_i^-\!\times\!\C\big)/\!\!\sim,
\qquad D_i^+\!\times\!\C\ni(x,0,v)\sim\big(x,0,f_i(x)v\big)\in D_i^-\!\times\!\C,\\
\wt\si_i'\big([q_{\Si}(x,\tau),v]\big)=\big[q_{\Si}(x,-\tau),\bar{v}\big].
\end{gather*}
We will call the restriction of a rank~$n$ real bundle pair $(V',\vph')$ 
over~$(\Si',\si')$ to a central component~$S^2_i$ \sf{standard} if
it equals either 
$$\big(S^2_i\!\times\!\C^n,\si'\!\times\!\fc\big) \qquad\hbox{or}\qquad
\big(\ga_i,\wt\si_i'\big)\oplus\big(S^2_i\!\times\!\C^{n-1},\si'\!\times\!\fc\big);$$
the latter is a possibility only if $|c_i|\!=\!0$.

\begin{lmm}\label{bndpull_lmm}
Let $(\Si,\si)$ be a smooth symmetric surface.
For every real bundle pair~$(V,\vph)$ over~$(\Si,\si)$,
there exists a real bundle pair  $(V',\vph')$ over~$(\Si',\si')$
such that $(V,\vph)$ is isomorphic to~$q_{\Si}^*(V',\vph')$ and the restriction of
$(V',\vph')$ to each central component~$(S^2_i,\si_i)$ is 
a standard real bundle pair.
\end{lmm}

\begin{proof}
Let $n\!=\!\rk_{\C}V$.
If $|c_i|\!=\!0$, define a rank~1 real bundle pair $(\wt\ga_i,\wt\si_i)$ over~$(U_i,\si)$~by
\begin{gather*}
\wt\ga_i\equiv \big((\prt\Si^b)_i\!\times\![0,2)\!\times\!\C \sqcup
(\prt\Si^b)_i\!\times\!(-2,0]\!\times\!\C\big)/\!\!\sim,\\
\qquad (\prt\Si^b)_i\!\times\![0,2)\!\times\!\C\ni(x,0,v)\sim
\big(x,0,f_i(x)v\big)\in(\prt\Si^b)_i\!\times\!(-2,0]\!\times\!\C,\\
\wt\si_i(x,\tau,v)=\big(x,-\tau,\bar{v}\big).
\end{gather*}
If in addition the restriction of $V^{\vph}$ to $(\prt\Si^b)_i$ is not orientable,
then 
$$V^{\vph}\approx
\big\{(x,v)\!\in\!(\prt\Si^b)_i\!\times\!\C\!:\,f_i(x)v\!=\!\bar{v}\big\} 
\oplus(\prt\Si^b)_i\!\times\!\R^{n-1}$$
and thus 
$$(V,\vph)\big|_{U_i}\approx (\wt\ga_i,\wt\si_i)\!\oplus\!
\big(U_i\!\times\!\C^{n-1},\id\!\times\!\fc\big)$$
as real bundle pairs over~$(U_i,\si)$.
If the restriction of $V^{\vph}$ to $(\prt\Si^b)_i$ is instead orientable,
then $V^{\vph}|_{(\prt\Si^b)_i}\!\approx\!(\prt\Si^b)_i\!\times\!\R^n$ and thus
\BE{bndpull_e5}(V,\vph)\big|_{(\prt\Si^b)_i}\approx 
\big((\prt\Si^b)_i\!\times\!\C^n,\id\!\times\!\fc\big)
=\big((\prt\Si^b)_i\!\times\!\C^n,\si'\!\times\!\fc\big).\EE
as real bundle pairs over~$((\prt\Si^b)_i,\si)$.
If $|c_i|\!=\!1$, then
\BE{bndpull_e3}(V,\vph)\big|_{(\prt\Si^b)_i}\approx 
\big((\prt\Si^b)_i\!\times\!\C^n,\si\!\times\!\fc\big)\EE
as real bundle pairs over $((\prt\Si^b)_i,\si)\!=\!((\prt\Si^b)_i,\fa)$; 
see \cite[Lemma 2.4]{Teh}.
The isomorphisms~\eref{bndpull_e5} and~\eref{bndpull_e3} extend to isomorphisms
$$(V,\vph)\big|_{U_i}\approx 
\big(U_i\!\times\!\C^n,\si\!\times\!\fc\big)$$
of real bundle pairs over~$(U_i,\si)$.
In all three cases, the restriction $(V,\vph)$ to the union of small neighborhoods
 of the pinching circles $\tau\!=\!\pm1$ in each~$U_i$ is  trivialized as a real bundle pair.
Therefore, $(V,\vph)$ descends to a real bundle pair $(V',\vph')$ over~$(\Si',\si')$
such that $(V,\vph)$ is isomorphic to~$q_{\Si}^*(V',\vph')$.
By construction, the restriction of
$(V',\vph')$ to each central component~$(S^2_i,\si_i)$ is 
a standard real bundle pair.
\end{proof}

\begin{proof}[{\bf\emph{Proof of Proposition~\ref{BHH_prp}}}] 
The necessity of the conditions is clear. 
By Lemma~\ref{bndpull_lmm}, it thus remains to show that real bundle pairs 
$(V_1',\vph_1')$ and $(V_2',\vph_2')$ over~$(\Si',\si')$  that
restrict to the same standard  real bundle pair on each central component~$S^2_i$
and to  bundles of the same degree over~$\Si'_+$ are isomorphic as real bundle 
pairs over~$(\Si',\si')$.
The identifications of   $(V_1',\vph_1')$ and $(V_2',\vph_2')$ over the central components
determine identifications of the restrictions of $V_1'|_{\Si'_{\pm}}$ and 
$V_2'|_{\Si'_{\pm}}$ at the nodes carried by~$\Si'_{\pm}$ that commute with~$\vph_1'$
and~$\vph_2'$.
Since $V_1'|_{\Si'_+}$ and $V_2'|_{\Si'_+}$ are complex bundles of the same degree
and rank and $\GL_n\C$ is connected,
we can choose an isomorphism~$\Psi_+'$ between them that respects the identifications
at the nodal points.
Let 
$$\Psi_-'= \vph_2\!\circ\!\Psi_+'\!\circ\!\vph_1\!:  
V_1'|_{\Si'_-}\lra V_2'|_{\Si'_-}\,.$$
This isomorphism again respects  the identifications at the nodal points.
We~take 
$$\Psi'\!: (V_1',\vph_1')\lra (V_2',\vph_2')$$
to be the identity on the central components of~$\Si'$ and $\Psi_{\pm}'$ on~$\Si'_+$.
This is a well-defined isomorphism of real bundle pairs.
\end{proof}

\section{Construction of automorphisms}
\label{Isom_sec}

\noindent
In this section, we establish the surjectivity of the homomorphism~\eref{SLrest_e}
and its $\GL(V,\vph)$ version and show that 
there is no obstruction to lifting paths with basepoints.

\begin{prp}\label{IsomExtend_prp}
Let $(\Si,\si)$ be a symmetric surface and $(V,\vph)$ be a real bundle pair
over~$(\Si,\si)$.
Then the homomorphism~\eref{SLrest_e}
is surjective.
Furthermore, every path $\psi_t$ in $\SL(V^{\vph})$ passing through 
$\Psi|_{V^{\vph}}$ for some $\Psi\!\in\!\SL(V,\vph)$ lifts to a path 
in $\SL(V,\vph)$ passing through~$\Psi$. 
The same is the case with $\SL(V^{\vph})$ and $\SL(V,\vph)$ replaced by 
$\GL(V^{\vph})$ and $\GL(V,\vph)$, respectively.
In all cases, the lifts can be chosen to restrict to the identity outside
of an arbitrary small neighborhood of~$\Si^{\si}$. 
\end{prp}

\begin{lmm}\label{IsomExtend_lmm}
Let $(\Si,\si)$ be a smooth symmetric surface with fixed components $\Si^{\si}_1,\ldots,\Si^{\si}_m$
and $(V,\vph)$ be a real bundle pair over~$(\Si,\si)$.
For every $i=1,\ldots,m$ and a path $\psi_t$ in $\SL(V^{\vph}|_{\Si^{\si}_i})$,
there exists a path $\Psi_t$ in $\SL(V,\vph)$ such that each $\Psi_t$ is the identity outside 
of an arbitrarily small neighborhood of $\Si^{\si}_i$ and
restricts to~$\psi_t$ on $V^{\vph}|_{\Si^{\si}_i}$.
The same is the case with $\SL(V^{\vph}|_{\Si^{\si}_i})$ and $\SL(V,\vph)$ replaced by 
$\GL(V^{\vph}|_{\Si^{\si}_i})$ and $\GL(V,\vph)$, respectively.
\end{lmm}

\begin{proof}
Let $n\!=\!\rk_{\C}\!V$ and $\bI\!=\![0,1]$.
Since every complex vector bundle over $\Si^{\si}_i$ is trivial,
$$\psi_t\in  \SL\big(V^{\vph}|_{\Si^{\si}_i}\big)\subset \SL\big(V|_{\Si^{\si}_i}\big)$$
determines a path of loops in~$\SL_n\C$.
Since $\pi_1(\SL_n\C)$ is trivial,
there exists a continuous~map
$$H\!:\bI^2\lra\SL\big(V|_{\Si^{\si}_i}\big),
~~(t,\tau)\lra H_{t,\tau}, \qquad\hbox{s.t.}\quad
H_{t,0}=\psi_t,~~H_{t,1}\!=\!\Id_{V|_{\Si^{\si}_i}}\quad\forall~t\!\in\!\bI.$$
Let $\Si^{\si}_i\!\times\!(-2,2)\!\lra\!\Si$ be a parametrization of a neighborhood~$U$ 
of $\Si^{\si}_i$ such that $\Si^{\si}_i\!\times\!0$ corresponds to~$\Si^{\si}_i$ and
$$\si(x,\tau)=(x,-\tau) \qquad\forall~(x,\tau)\!\in\!\Si^{\si}_i\!\times\!(-2,2).$$ 
Identifying $(V,\vph)|_U$ with $V|_{\Si^{\si}_i}\!\times\!(-2,2)$,
we define~$\Psi_t$ on~$U$~by   
$$\Psi_t|_{(x,\tau)}=
\begin{cases} 
H_{t,\tau}|_x, &\text{if}~ \tau\!\in\![0,1];\\
\Id_{V|_{\Si^{\si}_i}}, &\text{if}~\tau\!\in\![1,2);\\
\vph\!\circ\!H_{t,-\tau}|_x\!\circ\!\vph,&\text{if}~\tau\!\in\!(-2,0];\\
\end{cases}$$
and extend it as the identity over $\Si\!-\!U$.\\

\noindent
A similar argument applies with $\SL(V^{\vph}|_{\Si^{\si}_i})$ and $\SL(V,\vph)$ replaced by 
$\GL(V^{\vph}|_{\Si^{\si}_i})$ and $\GL(V,\vph)$, respectively.
A path $\psi_t$ in $\GL(V^{\vph}|_{\Si^{\si}_i})$ determines a path of loops~in
$$\big\{\psi\!\in\!\GL_n\C\!:\,{\det}_{\C}\!\psi\!\in\!\R^*\big\}.$$
These loops are again contractible, and the remainder of the above reasoning still applies. 
\end{proof}

\begin{crl}\label{IsomExtend_crl}
The first statement Proposition~\ref{IsomExtend_prp}, 
its analogue as in the third statement,
and its sharpening as in the fourth statement 
hold if $(\Si,\si)$ is a smooth symmetric surface.
\end{crl}

\begin{proof}
Let $\Si^{\si}_1,\ldots,\Si^{\si}_m$ be the connected components of the fixed locus 
$\Si^{\si}\!\subset\!\Si$.
If $\psi\!\in\!\SL(V^{\vph})$, let $\Psi_i\!\in\!\SL(V,\vph)$ be an automorphism
as in Lemma~\ref{IsomExtend_lmm} corresponding to  the restriction of~$\psi$ to $V^{\vph}|_{\Si^{\si}_i}$
and define
$$\Psi=\Psi_1\circ\ldots\circ\Psi_m\in\SL(V,\vph).$$
Since each $\Psi_i$ is the identity outside of a small neighborhood of~$\Si^{\si}_i$,
$\Psi|_{V^{\vph}}\!=\!\psi$.\\

\noindent
The same arguments apply with $\SL(V^{\vph})$ and $\SL(V,\vph)$ replaced by 
$\GL(V^{\vph})$ and $\GL(V,\vph)$, respectively.
\end{proof}

\noindent
A nodal oriented surface~$\Si$ is obtained from a smooth  oriented surface~$\wt\Si$
by identifying the two points in each of finitely many disjoint pairs of points of~$\wt\Si$;
the images of these pairs in~$\Si$ are the \sf{nodes} of~$\Si$.
The surface~$\wt\Si$ is called \sf{the normalization of~$\Si$};
it is unique up to a diffeomorphism respecting the distinguished pairs of points.
An orientation-reversing involution~$\si$ on~$\Si$ lifts to 
an orientation-reversing involution~$\wt\si$ on~$\wt\Si$.
There are three distinct types of nodes a nodal symmetric surface~$(\Si,\si)$ may~have
\begin{enumerate}[leftmargin=.3in]

\item[(H)]\label{Hnode_it}  
a \sf{non-isolated real node} $x_{ij}$,  i.e.~$x_{ij}$ is not an isolated point of~$\Si^{\si}$
and is obtained by identifying distinct points $\wt{x}_i,\wt{x}_j\!\in\!\wt\Si^{\wt\si}$;

\item[(E)] an \sf{isolated real node} $x_i$, i.e.~$x_i$ is an isolated point of~$\Si^{\si}$
and is obtained by identifying a point $\wt{x}_i^+\!\in\!\wt\Si\!-\!\wt\Si^{\wt\si}$ 
with $\wt{x}_i^-\!=\!\wt\si(\wt{x}_i^+)$;

\item[(C)] a pair $\{x_{ij}^+,x_{ij}^-\}$ of \sf{conjugate nodes}, 
i.e.~$x_{ij}^{\pm}\!\not\in\!\Si^{\si}$ and $x_{ij}^-\!=\!\si(x_{ij}^+)$,
with $x_{ij}^{\pm}$ obtained by identifying distinct points 
$\wt{x}_i^{\pm},\wt{x}_j^{\pm}\!\in\!\wt\Si\!-\!\wt\Si^{\wt\si}$ 
such that $\wt{x}_i^-\!=\!\wt\si(\wt{x}_i^+)$ and $\wt{x}_j^-\!=\!\wt\si(\wt{x}_j^+)$.

\end{enumerate} 

\begin{proof}[{\bf\emph{Proof of Proposition~\ref{IsomExtend_prp}}}]
Let $\wt\Si\!\subset\!\Si$ be the normalization of~$\Si$ and
$S_H,S_E,S_C\!\subset\!\wt\Si$ be the preimages of the $H$, $E$, and~$C$ nodes of~$\Si$,
respectively.
Choose disjoint subsets $S_1,S_2\!\subset\!\wt\Si$ consisting of one point from
the preimage of each node of~$\Si$.
Let $\wt{V}\!\lra\!\wt\Si$ be a complex vector bundle and
$$\vt\!:\wt{V}\big|_{S_1}\lra \wt{V}\big|_{S_2}$$
be an isomorphism of complex vector bundles such~that 
$$V=\wt{V}\big/\!\!\sim, ~~ v\!\sim\!\vt(v)~\forall\,v\!\in\!\wt{V}\big|_{S_1}.$$
Denote by $\wt\vph$ the lift of $\vph$ to~$\wt{V}$.
Thus,  $(\wt{V},\wt\vph)$ is a real bundle pair over $(\wt\Si,\wt\si)$
that descends to the real bundle pair $(V,\vph)$.
An automorphism~$\wt\Psi$ of $(\wt{V},\wt\vph)$ descends to an automorphism of
$(V,\vph)$ if and only~if
\BE{IsomExtend_e5}\wt\Psi\circ\vt=\vt\circ\wt\Psi\big|_{\wt{V}|_{S_1}}\,.\EE\\

\noindent
An automorphism $\psi\!\in\!\SL(V^{\vph})$ induces automorphisms
\begin{alignat}{3}
\label{IsomExtend_e7a}
\wt\psi_{\R}&\in\SL\big(\wt{V}^{\wt\vph}\big) &\quad\hbox{s.t.}\quad
\wt\psi_{\R}\circ\vt&=\vt\circ\wt\psi_{\R}
&\quad\hbox{on}~~\wt{V}\big|_{S_1\cap S_H},\\
\label{IsomExtend_e7b}
\wt\psi_E&\in \SL\big(\wt{V}|_{S_E}\big) &\quad\hbox{s.t.}\quad
\wt\psi_E\circ\vt&=\vt\circ\wt\psi_E,~~
\wt\psi_E\circ\wt\vph=\wt\vph\circ\wt\psi_E
&\quad\hbox{on}~~\wt{V}\big|_{S_1\cap S_E}\,.
\end{alignat}
By Corollary~\ref{IsomExtend_crl}, there exists $\wt\Psi'\!\in\!\SL(\wt{V},\wt\vph)$
such that $\wt\Psi'|_{\wt{V}^{\vph}}\!=\!\wt\psi_{\R}$ and $\wt\Psi'$ restricts to
the identity outside of an arbitrarily small neighborhood of~$\wt\Si^{\wt\si}$.
By the second condition on~$\wt\Psi'$, we can assume that $\wt\Psi'$ restricts
to the identity over~$S_C$ and over a collection of small disjoint neighborhoods
$U_{\wt{x}}$ of the elements $\wt{x}\!\in\!S_E$ not intersecting~$S_C$.
By~\eref{IsomExtend_e7a} and the first condition on~$\wt\Psi'$,
\BE{IsomExtend_e9}
\wt\Psi'\circ\vt
=\vt\circ\wt\Psi'\big|_{\wt{V}|_{S_1}}\,.\EE
By Lemma~\ref{IsomCextend_lmm} below, for each $\wt{x}\!\in\!S_1\!\cap\!S_E$  
there exists 
$\wt\Psi_{\wt{x}}\!\in\!\SL(\wt{V},\wt\vph)$ such that $\wt\Psi_{\wt{x}}|_{\wt{x}}\!=\!\wt\psi_E$
and $\wt\Psi_{\wt{x}}\!=\!\Id_{\wt{V}}$ outside of~$U_{\wt{x}}\!\cup\!U_{\wt\si(\wt{x})}$.
By~\eref{IsomExtend_e7b} and the two conditions on~$\wt\Psi_{\wt{x}}$,
\BE{IsomExtend_e11}
\wt\Psi_{\wt{x}}\circ\vt
=\vt\circ\wt\Psi_{\wt{x}}\big|_{\wt{V}|_{S_1}}\,.\EE
Let $\wt\Psi$ be the composition of the automorphisms $\wt\Psi'$ and
$\wt\Psi_{\wt{x}}$ with $\wt{x}\!\in\!S_1\!\cap\!S_E$.
Since the subsets of~$\wt\Si$ where these automorphisms differ from the identity are disjoint,
$\wt\Psi$ does not depend on their ordering in this composition and satisfies
\BE{IsomExtend_e15}
\wt\Psi\big|_{\wt{V}^{\vph}}=\wt\psi_{\R}\,, \qquad
\wt\Psi\big|_{\wt{V}|_{S_1\cap S_E}}=\wt\psi_E\,.\EE
By~\eref{IsomExtend_e5}, \eref{IsomExtend_e9}, and~\eref{IsomExtend_e11}, 
$\wt\Psi$ descends to an element $\Psi\!\in\!\SL(V,\vph)$.
By~\eref{IsomExtend_e15}, the latter satisfies $\Psi|_{V^{\vph}}\!=\!\psi$.\\

\noindent
Suppose $\Psi\!\in\!\SL(V,\vph)$ and $\psi_t$ is a path in $\SL(V^{\vph})$ 
such that $\Psi|_{V^{\vph}}\!=\!\psi_0$.
Let $\Phi_t\!\in\!\SL(V,\vph)$ be a path of automorphisms 
with $\Phi_t|_{V^{\vph}}\!=\!\psi_t$ constructed as above
and define
$$\Psi_t=\Psi\!\circ\!\Phi_0^{-1}\!\circ\!\Phi_t.$$
Since $\Psi|_{V^{\vph}}\!=\!\psi_0$ and $\Phi_t|_{V^{\vph}}\!=\!\psi_t$,
$\Psi_t|_{V^{\vph}}\!=\!\psi_t$.\\  

\noindent
The same arguments apply with $\SL(V^{\vph})$ and $\SL(V,\vph)$ replaced by 
$\GL(V^{\vph})$ and $\GL(V,\vph)$, respectively.
\end{proof}

\begin{lmm}\label{IsomCextend_lmm}
Let $(\Si,\si)$ be a symmetric surface and $(V,\vph)$ be a real bundle pair
over~$(\Si,\si)$.
For every $x\!\in\!\Si\!-\!\Si^{\si}$, an open neighborhood $U\!\subset\!\Si$ of~$x$, 
and a path $\psi_{t;x}\!\in\!\SL(V_x)$,
there exists a path $\Psi_t\!\in\!\SL(V,\vph)$ such that $\Psi_t|_x\!=\!\psi_{t;x}$
and $\Psi_t\!=\!\Id$ on $\Si\!-\!U\!\cup\!\si(U)$.
The same is the case with $\SL(V_x)$ and $\SL(V,\vph)$ replaced by 
$\GL(V_x)$ and $\GL(V,\vph)$, respectively.
\end{lmm}

\begin{proof}
By shrinking~$U$, we can assume that $U\!\cap\!\si(U)\!=\!\eset$
and that $V|_U\!=\!U\!\times\!V_x$.
Let $\rho\!:\Si\!\lra\![0,1]$ be a smooth $\si$-invariant function such that 
$\rho(x)\!=\!0$ and $\rho\!=\!1$ on $\Si\!-\!U\!\cup\!\si(U)$.
Since $\SL(V_x)$ is connected, there exists a continuous~map
$$H\!:\bI^2\lra\SL(V_x),
~~(t,\tau)\lra H_{t,\tau}, \quad\hbox{s.t.}~~~
H_{t,0}=\psi_{t;x},~~H_{t,1}\!=\!\Id_{V_x}\quad\forall~t\!\in\!\bI.$$
The path $\Psi_t\!\in\!\SL(V,\vph)$ given~by
$$\Psi_t(z)=\begin{cases} H_{t,\rho(z)},&\hbox{if}~z\!\in\!U;\\
 \vph\!\circ\!H_{t,\rho(z)}\!\circ\!\vph,&\hbox{if}~z\!\in\!\si(U);\\
\Id_{V_x},&\hbox{if}~z\!\not\in\!U\!\cup\!\si(U);
\end{cases}$$
has the desired properties.
The same arguments apply with $\SL(V_x)$ and $\SL(V,\vph)$ replaced by 
$\GL(V_x)$ and $\GL(V,\vph)$, respectively.
\end{proof}

\section{Construction of isomorphisms}
\label{IsomExist_sec}

\noindent
We will next establish Theorem~\ref{nodalBHH_thm} by induction on the number
of nodes from the base case of Proposition~\ref{BHH_prp}.
We will need the following lemma.

\begin{lmm}\label{LinAlg_lmm}
Let $(V,\fI)$ be a finite-dimensional complex vector space and $A,B\!:V\!\lra\!V$
be \hbox{$\C$-antilinear} isomorphisms such that $A^2,B^2\!=\!\Id_V$.
Then there exists a $\C$-linear isomorphism $\psi\!:V\!\lra\!V$ such that 
$\psi\!=\!A\!\circ\!\psi\!\circ\!B$.
If 
\BE{LinAlg_e}\big\{\La_{\C}^{\top}A\big\}\!\circ\!\big\{\La_{\C}^{\top}B\big\}
=\big\{\La_{\C}^{\top}B\big\}\!\circ\!\big\{\La_{\C}^{\top}A\big\}\!:
\La_{\C}^{\top}V\lra \La_{\C}^{\top}V,\EE
then $\psi$ can be chosen so that $\La_{\C}^{\top}\psi\!=\!\Id$.
\end{lmm}

\begin{proof}
Since $A^2,B^2\!=\!\Id_V$, the isomorphisms $A,B$ are diagonalizable with all eigenvalues $\pm1$.
Since $A,B$ are $\C$-antilinear, we can choose $\C$-bases $\{v_i\}$ and $\{w_i\}$ for~$V$
such~that 
$$A(v_i)=v_i, \quad A(\fI v_i)=-\fI v_i, \qquad B(w_i)=w_i, \quad B(\fI w_i)=-\fI w_i.$$
The $\C$-linear isomorphism $\psi\!:V\!\lra\!V$ defined by $\psi(w_i)\!=\!v_i$
then has the first desired property.\\

\noindent
The automorphisms $\La_{\C}^{\top}A$ and $\La_{\C}^{\top}B$ are $\C$-antilinear and
have one eigenvalue of $+1$ and one of~$-1$.
If~\eref{LinAlg_e} holds, the eigenspaces of $\La_{\C}^{\top}A$ and $\La_{\C}^{\top}B$ 
are the same and~so
$$v_1\!\w_{\C}\!\ldots\!\w_{\C}\!v_n= r\cdot w_1\!\w_{\C}\!\ldots\!\w_{\C}\!w_n
\in \La_{\C}^{\top}V$$
for some $r\!\in\!\R^*$.
Replacing $w_1$ by $rw_1$ in the previous paragraph, we obtain an isomorphism~$\psi$
that also satisfies the second property.
\end{proof}

\begin{proof}[{\bf\emph{Proof of Theorem~\ref{nodalBHH_thm}}}]
By Proposition~\ref{BHH_prp}, we can assume that $(\Si,\si)$ is singular and 
that Theorem~\ref{nodalBHH_thm} holds for all symmetric surfaces~$(\wt\Si,\wt\si)$
with fewer nodes than~$(\Si,\si)$.\\

\noindent
If $(\Si,\si)$ contains a real node~$x$, i.e.~an $H$ or~$E$ node 
as described on page~\pageref{Hnode_it}, 
let $(\wt\Si,\wt\si)$ be the (possibly nodal) symmetric surface obtained from
$(\Si,\si)$ by desingularizing a small neighborhood of~$x$.
The preimage of~$x$ under the natural projection $\wt\Si\!\lra\!\Si$ is then two distinct
points $\wt{x}_1,\wt{x}_2$; $(\Si,\si)$ is obtained from $(\wt\Si,\wt\si)$
by identifying $\wt{x}_1$ with~$\wt{x}_2$ to form the additional node~$x$.
Let $S_1\!=\!\{\wt{x}_1\}$ and $S_2\!=\!\{\wt{x}_2\}$ in this case.
If $(\Si,\si)$ does not contain a real node,  
let $\{x^+,x^-\}$ be a pair of conjugate nodes and 
$(\wt\Si,\wt\si)$ be the (possibly nodal) symmetric surface obtained from
$(\Si,\si)$ by desingularizing small neighborhoods of~$x^+$ and~$x^-$.
The preimage of each point~$x^{\pm}$ under the natural projection $\wt\Si\!\lra\!\Si$ 
is then two distinct points $\wt{x}_1^{\pm},\wt{x}_2^{\pm}$; 
$(\Si,\si)$ is obtained from $(\wt\Si,\wt\si)$
by identifying $\wt{x}_1^+$ with~$\wt{x}_2^+$  and $\wt{x}_1^-$ with~$\wt{x}_2^-$ 
to form the additional nodes~$x^+$ and~$x^-$, respectively.
Let $S_1\!=\!\{\wt{x}_1^+,\wt{x}_1^-\}$ and $S_2\!=\!\{\wt{x}_2^+,\wt{x}_2^-\}$ in this case.\\

\noindent 
Let $\wt{V}_1,\wt{V}_2\!\lra\!\wt\Si$ be complex vector bundles and
$$\vt_1\!:\wt{V}_1\big|_{S_1}\lra \wt{V}_1\big|_{S_2}
\qquad\hbox{and}\qquad 
\vt_2\!:\wt{V}_2\big|_{S_1}\lra \wt{V}_2\big|_{S_2}$$
be isomorphisms of complex vector bundles such~that 
$$V_1=\wt{V}_1\big/\!\!\sim, ~~ v\!\sim\!\vt_1(v)~\forall\,v\!\in\!\wt{V}_1\big|_{S_1},
\qquad\hbox{and}\qquad 
V_2=\wt{V}_2\big/\!\!\sim, ~~ v\!\sim\!\vt_2(v)~\forall\,v\!\in\!\wt{V}_2\big|_{S_1}.$$
Denote by $\wt\vph_1$ and~$\wt\vph_2$ the lift of $\vph_1$ to~$\wt{V}_1$ and 
the lift of $\vph_2$ to~$\wt{V}_2$, respectively.
Thus,  $(\wt{V}_1,\wt\vph_1)$ and $(\wt{V}_2,\wt\vph_2)$ are real bundle pairs over $(\wt\Si,\wt\si)$
that descend to the real bundle pairs $(V_1,\vph_1)$ and $(V_2,\vph_2)$ over~$(\Si,\si)$.
Furthermore,
\BE{nodalBHH_e3}
\vt_i(v)=\begin{cases} \wt\vph_i(\vt_i^{-1}(\wt\vph_i(v))),
&\hbox{if}~|S_1|\!=\!1~\hbox{and}~x~\hbox{is $E$ node};\\
\wt\vph_i(\vt_i(\wt\vph_i(v))), &\hbox{otherwise};
\end{cases}\EE
for all $v\!\in\!\wt{V}_i\big|_{S_1}$.\\

\noindent
Since $(\wt\Si,\wt\si)$ satisfies Theorem~\ref{nodalBHH_thm},
there exists an isomorphism 
$$\wt\Phi\!:(\wt{V}_1,\wt\vph_1)\lra (\wt{V}_2,\wt\vph_2)$$
of real bundle pairs over $(\wt\Si,\wt\si)$.
We show below that there exists
\hbox{$\wt\Psi\!\in\!\GL(\wt{V}_1,\wt\vph_1)$}
so~that 
\BE{nodalBHH_e2b}\wt\Phi\!\circ\!\wt\Psi\circ\vt_1=\vt_2\circ\wt\Phi\!\circ\!\wt\Psi\!:
\wt{V}_1|_{S_1}\lra \wt{V}_2|_{S_2}\,.\EE
This implies that $\wt\Phi\!\circ\!\wt\Psi$
 descends to an isomorphism of real bundles as in~\eref{nodalBHH_e}.\\

\noindent
Suppose $|S_1|\!=\!1$ and $x$ is an $E$ node.
By the first case in~\eref{nodalBHH_e3}, the $\C$-linear isomorphisms
$$\wt\Phi^{-1}\!\circ\!\vt_2^{-1}\!\circ\!\wt\Phi\!\circ\!\wt\vph_1
\!=\!\wt\Phi^{-1}\!\circ\!\vt_2^{-1}\!\circ\!\wt\vph_2\!\circ\!\wt\Phi,
\wt\vph_1\!\circ\!\vt_1\!: \wt{V}_1\big|_{\wt{x}_1}\lra \wt{V}_1\big|_{\wt{x}_1}$$
square to the identity.
By Lemma~\ref{LinAlg_lmm},
there thus exists $\psi\!\in\!\GL(\wt{V}_1|_{\wt{x}_1})$ such~that 
\BE{nodalBHH_e4a} \psi=\wt\Phi^{-1}\!\circ\!\vt_2^{-1}\!\circ\!\wt\Phi\!\circ\!\wt\vph_1
\circ\psi\circ \wt\vph_1\!\circ\!\vt_1\!:\,\wt{V}_1|_{\wt{x}_1}\lra \wt{V}_1|_{\wt{x}_1}\,.\EE
By Lemma~\ref{IsomCextend_lmm}, there exist $\wt\Psi\!\in\!\GL(\wt{V}_1,\wt\vph_1)$
and a neighborhood $U$ of $\wt{x}_1$ in~$\wt\Si$ such~that 
\BE{nodalBHH_e4b} \wt\Psi|_z=
\begin{cases}\psi,&\hbox{if}~z\!=\!\wt{x}_1;\\
\id,&\hbox{if}~z\!\not\in\!U\!\cup\!\si(U);
\end{cases}
\qquad U\!\cap\!\si(U)=\eset.\EE
By~\eref{nodalBHH_e4a} and~\eref{nodalBHH_e4b}, $\wt\Psi$ satisfies~\eref{nodalBHH_e2b}.\\

\noindent
In the two remaining cases, let
$$\psi\!=\!\wt\Phi^{-1}\!\circ\!\vt_2^{-1}\!\circ\!\wt\Phi\!\circ\!\vt_1\!:
\wt{V}_1\big|_{S_1}\lra \wt{V}_1\big|_{S_1}\,.$$
This $\C$-linear automorphism satisfies
\BE{nodalBHH_e7} \wt\Phi\big(\vt_1(v)\big)=\vt_2\big(\wt\Phi(\psi(v))\big)
\qquad\forall~v\!\in\!\wt{V}_1\big|_{S_1}\,.\EE
By the second case in~\eref{nodalBHH_e3}, $\psi\!\circ\!\wt\vph_1\!=\!\wt\vph_1\!\circ\!\psi$.\\

\noindent
Suppose $|S_1|\!=\!1$ and $x$ is an $H$ node.
If~$\wt{x}_1$ and $\wt{x}_2$ lie on different topological components
$\wt\Si^{\wt\si}_1,\wt\Si^{\wt\si}_2$ of~$\wt\Si^{\wt\si}$,
extend~$\psi$ to some $\wt\psi\!\in\!\GL(\wt{V}_1^{\wt\vph_1}|_{\wt\Si^{\wt\si}_1})$.
If~$\wt{x}_1$ and $\wt{x}_2$ lie on the same topological component
$\wt\Si_1^{\wt\si}$ of~$\wt\Si^{\wt\si}$, the $w_1$-assumption applied to 
either of the two circles in the connected component of~$\Si^{\si}$ 
containing~$x$ implies that~$\psi$ is orientation-preserving.
Therefore, it can be extended to some $\wt\psi\!\in\!\SL(\wt{V}_1^{\wt\vph_1}|_{\wt\Si_1^{\wt\si}})$
such that $\wt\psi$ is the identity over~$\wt{x}_2$.
In both cases, there exist $\wt\Psi\!\in\!\GL(\wt{V}_1,\wt\vph_1)$
and a neighborhood $U$ of $\wt\Si^{\wt\si}_1$ in~$\wt\Si$ such~that 
\BE{nodalBHH_e8} \wt\Psi|_z=
\begin{cases}\psi,&\hbox{if}~z\!=\!\wt{x}_1;\\
\id,&\hbox{if}~z\!\not\in\!U;
\end{cases}
\qquad U\!\cap\!\big(\wt\Si^{\wt\si}\!-\!\wt\Si_1^{\wt\si}\big)=\eset;\EE
see Proposition~\ref{IsomExtend_prp}.
By~\eref{nodalBHH_e7} and~\eref{nodalBHH_e8}, $\wt\Psi$ satisfies~\eref{nodalBHH_e2b}.\\

\noindent
Suppose $|S_1|\!=\!2$.
Since $\psi\!\circ\!\wt\vph_1\!=\!\wt\vph_1\!\circ\!\psi$, 
there exist $\wt\Psi\!\in\!\Aut(\wt{V},\wt\vph_1)$  
and a neighborhood $U$ of~$S_1$ in~$\wt\Si$ such~that 
\BE{canIsomExt_e4b} 
\wt\Psi|_z=
\begin{cases}\psi|_{S_1},&\hbox{if}~z\!\in\!S_1;\\
\Id_{\wt{V}_1},&\hbox{if}~z\!\not\in\!U;
\end{cases}
\qquad U\!\cap\!S_2=\eset;\EE
see Lemma~\ref{IsomCextend_lmm}.
By~\eref{nodalBHH_e7} and~\eref{canIsomExt_e4b},
$\wt\Psi$ satisfies~\eref{nodalBHH_e2b}. 
\end{proof}

\section{Homotopies between automorphisms}
\label{Homotopies_sec}

\noindent
In this section, we establish the main statement needed to lift homotopies
in Theorem~\ref{auto_thm}.

\begin{prp}\label{homotopextend_prp}
Let $(\Si,\si)$ be a symmetric surface, $(V,\vph)$ be a real bundle pair
over~$(\Si,\si)$, and $\Psi\!\in\!\SL(V,\vph)$.
If $\Psi|_{V^{\vph}}\!=\!\Id_{V^{\vph}}$, then $\Psi$ is homotopic 
to~$\Id_V$ through automorphisms $\Psi_t\!\in\!\SL(V,\vph)$
such that $\Psi_t|_{V^{\vph}}\!=\!\Id_{V^{\vph}}$.
\end{prp}

\begin{proof} Let $n\!=\!\rk_{\C}V$.
By Lemma~\ref{homotoid_lmm} below, we can assume that $\Psi|_x\!=\!\Id_{V_x}$ for 
every $C$~node $x\!\in\!\Si$.
Let 
$$\big(\wt\Si,\wt\si\big)\lra(\Si,\si), \qquad S_E,S_C\subset\wt\Si, \qquad\hbox{and}\quad
\wt\Psi\in\SL(\wt{V},\wt\vph)$$ 
be as in the proof of Proposition~\ref{IsomExtend_prp}.
Let $(\wt\Si^b,\wt{c})$ be an oriented sh-surface which doubles to~$(\wt\Si,\wt\si)$
so that the boundary $\prt\wt\Si^b$ of~$\wt\Si^b$ is disjoint from~$S_E\!\cup\!S_C$
and~set 
$$S^+=\big(S_E\!\cup\!S_C)\cap\wt\Si^b\,.$$ 
Since $\Psi|_x\!=\!\Id_{V_x}$ for every node $x\!\in\!\Si$,  
$\wt\Psi|_{\wt{x}}\!=\!\Id_{\wt{V}_{\wt{x}}}$ for every $\wt{x}\!\in\!S^+$.\\

\noindent
Let $\prt_1\wt\Si^b$ be the union of the boundary components of $(\wt\Si^b,\wt{c})$
with $|c_i|\!=\!1$.
Since $\det\wt\Psi\!=\!1$, 
the restriction of~$\wt\Psi$ to $(\wt{V},\wt\vph)|_{\prt_1\wt\Si^b}$ is homotopic to the identity 
through automorphisms \hbox{$\wt\Phi_t\!\in\!\SL((\wt{V},\wt\vph)|_{\prt_1\wt\Si^b})$};
see \cite[Lemma~2.4]{Teh}.
We extend~$\wt\Phi_t$ over $\wt\Si^b$ as follows. 
Let $\prt_1\wt\Si^b\!\times\!\bI\!\lra\!U$ be a parametrization of 
a (closed) neighborhood~$U$ 
of  $\prt_1\wt\Si^b\!\subset\!\wt\Si^b\!-\!S^+$ with coordinates~$(x,s)$. 
Identifying $\wt{V}|_U$ with $\wt{V}|_{\prt_1\wt\Si^b}\!\times\!\bI$, define
$$\wt\Psi_t\in\SL\big(\wt{V}|_{\wt\Si^b}\big)  \qquad\hbox{by}\quad
\wt\Psi_t|_z=\begin{cases}
\wt\Phi_{(1-s)t}|_x\circ 
\wt\Psi|_x^{-1},&\text{if}~ z=(x,s)\in U\approx \prt_1\wt\Si^b\!\times\!\bI;\\
\Id_{\wt{V}_z},& \text{if}~ z\in \wt\Si^b\!-\!U.
\end{cases}$$
Since $\wt\Psi_t|_{(x,1)}$ is the identity for all~$t$, this map is continuous. 
Moreover, $\wt\Psi_0|_z$ is the identity for all $z\!\in\!\wt\Si^b$ and 
$$\wt\Psi_t|_{(x,0)}=\wt\Phi_t|_x\circ \wt\Psi|_x^{-1}$$ 
is a homotopy between the identity and~$\wt\Psi^{-1}$ on~$\wt{V}|_{\prt_1\wt\Si^b}$.
Thus, $\wt\Psi_t\!\circ\!\wt\Psi$ is a homotopy from $\wt\Psi$ 
over~$\wt\Si^b$ extending~$\wt\Phi_t$.\\

\noindent 
Choose embedded non-intersecting paths~$\{C_i\}$ 
in $\wt\Si^b\!-\!S^+$ with endpoints on~$\prt\wt\Si^b$ 
which cut $\wt\Si^b$ into a disk~$D_0^2$; see Figure~\ref{arcs_fig}.
Choose an embedded path~$\ga_{\wt{x}}$ in~$D^2_0$
from each point $\wt{x}\!\in\!S^+$ to $\prt D_0^2$ so that 
these curves are pairwise disjoint and cut~$D_0^2$ 
into another disk~$D^2$.
Thus, $\prt D^2$ is subdivided into arcs each of which was contained in~$\prt\wt\Si^b$
before the two cuttings, or had precisely the two endpoints in common with~$\prt\wt\Si^b$,
or had one endpoint on~$\prt\wt\Si^b$ and the other in~$S^+$.
Furthermore, every point in~$S^+$ is an endpoint of some arc in~$\prt D^2$.\\

\begin{figure}
\begin{center}
\advance\topskip-10cm
\includegraphics[width=0.3\textwidth,height=150px]{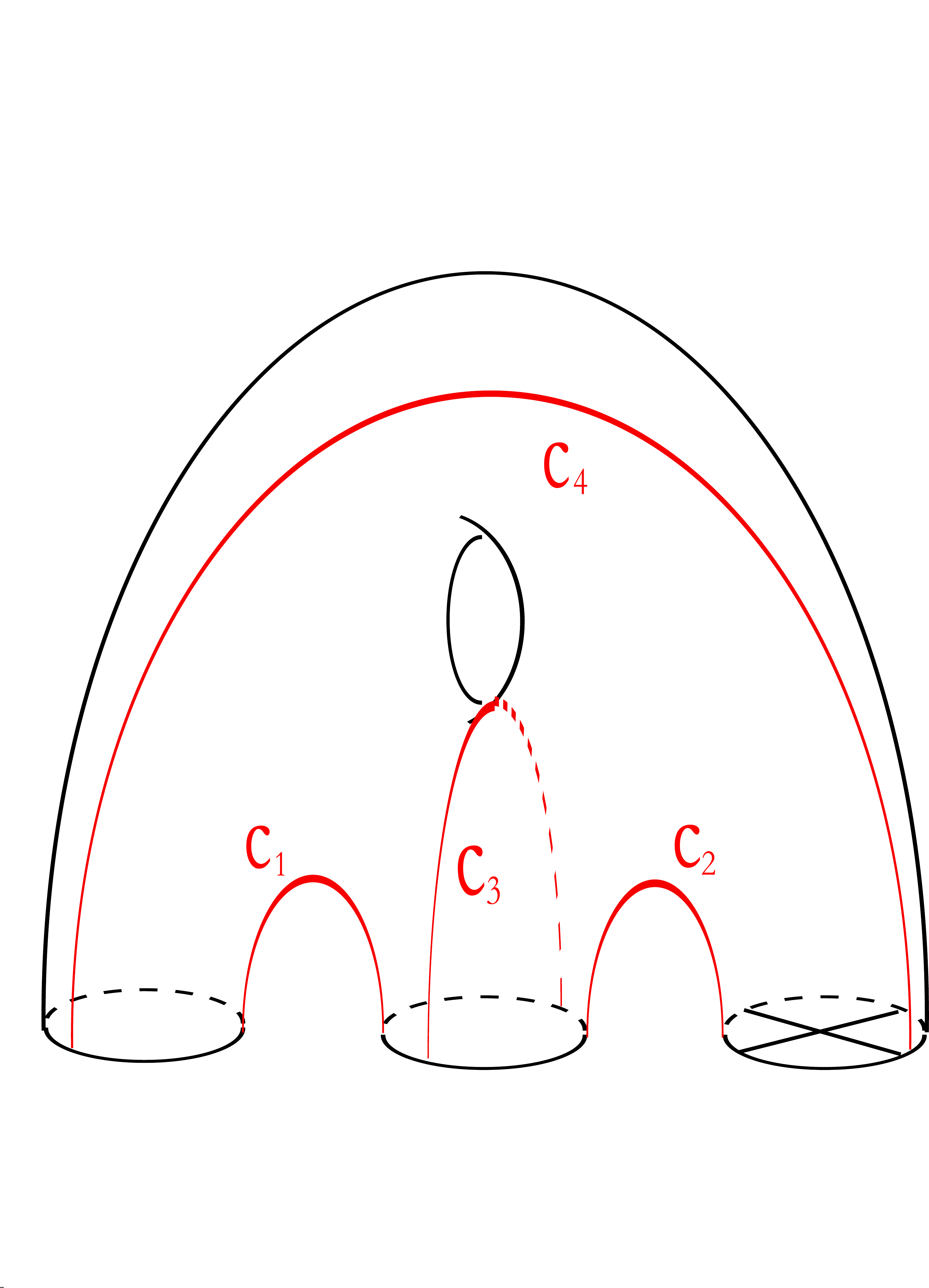}
\end{center}
\caption{The paths $C_1,\ldots,C_4$ cut $\wt\Si^b$ to a disk.}
\label{arcs_fig}
\end{figure}

\noindent
By the assumption  $\Psi|_{V^{\vph}}\!=\!\Id_{V^{\vph}}$ and 
the first two paragraphs, we may assume that $\wt\Psi$ is the identity 
over~$\prt\wt\Si^b$ and~$S^+$.
Since the restriction of~$\wt{V}$ to each~$C_i$ is trivial,
the restriction of~$\wt\Psi$ to~$C_i$ defines an element of 
\BE{pi1SL_e}\pi_1\big(\SL_n\C,I_n\big)\approx\pi_1\big(\SU_n,I_n\big)=0.\EE
Thus, we can homotope~$\wt\Psi$  to the identity over~$C_i$ 
while keeping it fixed at the endpoints. 
Similarly to the second paragraph, this homotopy extends over~$\wt\Si^b$ without 
changing~$\wt\Psi$ over~$\prt\wt\Si^b$, over~$C_j$ for any $C_j\!\neq\!C_i$,
or over~$S^+$.
In particular, this homotopy descends to~$\wt\Si^b$ restricting to 
the identity over~$\prt\wt\Si^b$ and~$S^+$.\\

\noindent
By the previous paragraph, we may assume that $\wt\Psi$ restricts 
to the identity over~$\prt D_0^2$. 
Since the restriction of~$\wt{V}$ to~$\ga_{\wt{x}}$ is trivial for each $\wt{x}\!\in\!S^+$,
the restriction of~$\wt\Psi$ to~$\ga_{\wt{x}}$ defines an element of 
$$\pi_1\big(\SL_n\C,I_n\big)\approx\pi_1\big(\SU_n,I_n\big)=0.$$
Thus, we can homotope~$\wt\Psi$ to the identity over~$\ga_{\wt{x}}$ while keeping it fixed 
at the endpoints. 
Similarly to the second paragraph, this homotopy extends over~$D^2_0$ without 
changing~$\wt\Psi$ over~$\prt D_0^2$ or over $\ga_{\wt{x}'}$ for any $\wt{x}'\!\in\!S^+$
different from~$\wt{x}$.
In particular, this homotopy descends to~$\wt\Si^b$ restricting to 
the identity over~$\prt\wt\Si^b$ and~$S^+$.\\

\noindent
By the previous paragraph, we may assume that $\wt\Psi$ restricts 
to the identity over the boundary~$S^1$ of~$D^2$. 
Since every vector bundle over~$D^2$ is trivial and 
$$\pi_2\big(\SL_n\C,I_n\big)\approx\pi_2\big(\SU_n,I_n\big)=0,$$ 
the map $\wt\Psi\!:(D^2,S^1)\!\lra\!(\SL_n\C,I_n)$
can be homotoped to the identity as a relative map.
Doubling such a homotopy~$\wt\Psi_t$  by the requirement that 
$\wt\Psi_t\!\circ\!\vph\!=\!\vph\!\circ\!\wt\Psi_t$,
we obtain a homotopy~$\wt\Psi_t$ from~$\wt\Psi$
to~$\Id_{\wt{V}}$ through automorphisms $\wt\Psi_t\!\in\!\SL(\wt{V},\wt\vph)$
such that $\wt\Psi_t$ restricts to the identity over~$\wt\Si^{\wt\si}$ and~$S^+$.
Thus, they descend to automorphisms $\Psi_t\!\in\!\SL(V,\vph)$
with the required properties.
\end{proof}

\begin{lmm}\label{homotoid_lmm}
Let $(\Si,\si)$ be a symmetric surface and $(V,\vph)$ be a real bundle pair
over~$(\Si,\si)$.
For every $x\!\in\!\Si\!-\!\Si^{\si}$, an open neighborhood $U\!\subset\!\Si$ of~$x$, 
and $\Psi\!\in\!\SL(V,\vph)$,
there exists a path $\Psi_t\!\in\!\SL(V,\vph)$ such that $\Psi_0\!=\!\Psi$,
$\Psi_1|_x\!=\!\Id_{V_x}$, and $\Psi_t\!=\!\Psi$ on $\Si\!-\!U\!\cup\!\si(U)$.
The same is the case with $\SL(V,\vph)$ replaced by~$\GL(V,\vph)$.
\end{lmm}

\begin{proof}
Since $\SL(V_x)$ is connected, there exists a~path
$$\psi_{x;t}\in \SL(V_x) \qquad\hbox{s.t.}\quad
\psi_{x;0}=\Id_{V_x},\quad \psi_{x;1}=\Psi^{-1}|_{V_x}\,.$$
By Lemma~\ref{IsomCextend_lmm},
there exists a path $\Phi_t\!\in\!\SL(V,\vph)$ such that $\Phi_t|_x\!=\!\psi_{x;t}$
and $\Phi_t\!=\!\Id$ on $\Si\!-\!U\!\cup\!\si(U)$.
The path $\Psi_t\!=\!\Phi_t\!\circ\!\Phi_0^{-1}\!\circ\!\Psi$ then has the desired properties.
\end{proof}

\section{Classification of automorphisms}
\label{MainPfs_sec}

\noindent
By the next lemma, the first composite homomorphism in~\eref{AutoThm_e} is surjective.
We use it to complete the proof of Theorem~\ref{auto_thm} in this section.

\begin{lmm}\label{RBisom_lmm0}
Let $(\Si,\si)$ be a symmetric surface 
and $(V,\vph)$ be a real bundle pair over~$(\Si,\si)$.
Then the homomorphism
$$\GL(V,\vph)\lra \cC(\Si,\si;\C^*), \qquad \Psi\lra \det\Psi,$$
is surjective.
Furthermore, every path $f_t$ in $\cC(\Si,\si;\C^*)$ passing through 
$\det\Psi$ for some $\Psi\!\in\!\GL(V,\vph)$ lifts to a path 
in $\GL(V,\vph)$ passing through~$\Psi$. 
\end{lmm}

\begin{proof}
Let $n\!=\!\rk_{\C}\!V$.
By Theorem~\ref{nodalBHH_thm}, we can assume~that
$$(V,\vph)=\La^{\top}_{\C}(V,\vph) \oplus 
\big(\Si\!\times\!\C^{n-1},\si\!\times\!\fc_{\C^{n-1}}\big).$$
If $f\!\in\!\cC(\Si,\si;\C^*)$, then 
$$\Psi \equiv f\Id_{\La^{\top}_{\C}V} \oplus\Id_{\Si\times\C^{n-1}}\!:(V,\vph)\lra(V,\vph)$$
is an element of $\GL(V,\vph)$ such that $\det\Psi\!=\!f$.
If 
$$\Psi\equiv f\Id_{\La^{\top}_{\C}V} \oplus \Phi$$
is an arbitrarily element of $\GL(V,\vph)$ and  
$f_t$ is a path in $\cC(\Si,\si;\C^*)$ such that $\det\Psi\!=\!f_0$,
then
$$\Psi_t\equiv \frac{f_t}{\det\Phi}\Id_{\La^{\top}_{\C}V} \oplus \Phi$$
is a path in $\GL(V,\vph)$ such that $\Psi_0\!=\!\Psi$ and $\det\Psi_t\!=\!f_t$.
\end{proof}

\begin{proof}[{\bf\emph{Proof of Theorem~\ref{auto_thm}}}]
By Lemma~\ref{RBisom_lmm0} and Proposition~\ref{IsomExtend_prp}, the compositions
\begin{alignat*}{2}
\GL(V,\vph)&\lra \GL'(V,\vph)\lra \cC(\Si,\si;\C^*), &\qquad
\Psi&\lra\det\Psi,\\
\GL(V,\vph)&\lra \GL'(V,\vph)\lra \GL(V^{\vph}), &\qquad
\Psi&\lra\Psi|_{V^{\vph}},
\end{alignat*}
are surjective.
This implies that the projection homomorphisms
$$\GL'(V,\vph)\lra \cC(\Si,\si;\C^*),\GL(V^{\vph})$$ are surjective.\\

\noindent
Suppose  $\Psi,\Phi\!\in\!\GL(V,\vph)$
are such that  $\det\Psi$ and $\det\Phi$ lie in the same path component
of  $\cC(\Si,\si;\C^*)$ and 
$\Psi|_{V^{\vph}}$ and $\Phi|_{V^{\vph}}$ lie in the same path component
of~$\GL(V^{\vph})$.
We will show that $\Psi$ and $\Phi$ lie in the same path component
of~$\GL(V,\vph)$.
By the second statement of Lemma~\ref{RBisom_lmm0}, we may assume~that 
$\det\Psi\!=\!\det\Phi$ and thus
$$\Th\equiv\Phi\!\circ\!\Psi^{-1}\in\SL(V,\vph), \qquad
\psi\!\equiv\!\Th|_{V^{\vph}}\in\SL(V^{\vph}).$$
Since $\Psi|_{V^{\vph}}$ and $\Phi|_{V^{\vph}}$ lie in the same path component
of~$\GL(V^{\vph})$, there exists a path $\psi_t$ in $\SL(V^{\vph})$
from $\psi_0\!\equiv\!\Id_{V^{\vph}}$ to $\psi_1\!\equiv\!\psi$.
By Proposition~\ref{IsomExtend_prp}, there exists a path $\Th_t$ in~$\SL(V,\vph)$
such that $\Th_t|_{V^{\vph}}\!=\!\psi_t$ and $\Th_1\!=\!\Th$.
In particular, $\Phi$ and $\Th_0\!\circ\!\Psi$ lie 
in the same path component of~$\GL(V,\vph)$,
$$\det(\Th_0\!\circ\!\Psi)=\det\Psi, \qquad\hbox{and}\qquad
(\Th_0\!\circ\!\Psi)\big|_{V^{\vph}}=\Psi|_{V^{\phi}}.$$
By Proposition~\ref{homotopextend_prp}, 
$\Psi$ lies in the same path component of~$\GL(V,\vph)$ as
$\Th_0\!\circ\!\Psi$ and~$\Phi$.\\

\noindent
Suppose  $\Psi,\Phi\!\in\!\GL(V,\vph)$ and $(f_t,\psi_t)$ is a path in $\GL'(V,\vph)$
such~that
$$(f_0,\psi_0)=\big(\det\Psi,\Psi|_{V^{\vph}}\big), \qquad
(f_1,\psi_1)=\big(\det\Phi,\Phi|_{V^{\vph}}\big).$$
By the second statement of Lemma~\ref{RBisom_lmm0}, there exists 
a path $\Phi_t\!\in\!\GL(V,\vph)$ such that $\Phi_0\!=\!\Psi$ and 
$\det\Phi_t\!=\!f_t$.
Let
$$\wh\psi_t=\psi_t\circ\{\Phi_t\}^{-1}|_{V^{\vph}}\in\GL(V^{\vph}).$$
Since $\det\psi_t\!=\!f_t|_{\Si^{\si}}$, $\wh\psi_t\!\in\!\SL(V^{\vph})$.
Furthermore, $\wh\psi_0\!=\!\Id|_{V^{\vph}}$.
By Proposition~\ref{IsomExtend_prp}, $\wh\psi_t$ thus extends to a path $\wh\Psi_t$
in~$\SL(V,\vph)$ such that $\wh\Psi_0\!=\!\Id_V$.
Let $\wt\Psi_t\!=\!\wh\Psi_t\!\circ\!\Phi_t$.
In particular, 
$$\wt\Psi_0=\Psi, \quad \det\wt\Psi_t=f_t, \quad \wt\Psi_t\big|_{V^{\vph}}=\psi_t, \qquad
\det\big(\Phi\!\circ\!\wt\Psi_1^{-1}\big)=\Id_V, \quad
\big(\Phi\!\circ\!\wt\Psi_1^{-1}\big)|_{V^{\vph}}=\Id_{V^{\vph}}.$$
By Proposition~\ref{homotopextend_prp}, the last two properties imply that there exists
a path~$\Th_t$ in~$\SL(V,\vph)$ from~$\Id_V$ to $\Phi\!\circ\!\wt\Psi_1^{-1}$
such that $\Th_t|_{V^{\vph}}\!=\!\Id_{V^{\vph}}$.
The path $\Psi_t\!\equiv\!\Th_t\!\circ\!\wt\Psi_t$ in~$\GL(V,\vph)$ runs 
from~$\Psi$ to~$\Phi$ and lifts~$(f_t,\psi_t)$.
\end{proof}

\section{Connections with real Gromov-Witten theory}
\label{realGW_sec}

\noindent
Let $X$ be a topological space.
For real vector bundles $V_1,V_2\!\lra\!X$ of the same rank, 
let $\Isom(V_1,V_2)$ be the space of vector bundle isomorphisms
$$\psi\!:V_1\lra V_2$$ 
covering~$\id_X$.
For any such isomorphism, let
$$\La_{\R}^{\top}\psi\!:\La^{\top}_{\R}V_1\lra \La^{\top}_{\R}V_2$$
be the induced element of $\Isom(\La^{\top}_{\R}V_1,\La^{\top}_{\R}V_2)$.\\

\noindent
Let $(X,\phi)$ be a topological space with an involution.
For real bundle pairs~$(V_1,\vph_1)$ and~$(V_2,\vph_2)$ over~$(X,\phi)$
satisfying the conditions of Theorem~\ref{nodalBHH_thm},
let $\Isom((V_1,\vph_1),(V_2,\vph_2))$ be the space of isomorphisms
$$\Psi\!:(V_1,\vph_1)\lra (V_2,\vph_2)$$ 
of real bundle pairs covering~$\id_X$.
For any such isomorphism, let
$$\La_{\C}^{\top}\Psi\!:\La^{\top}_{\C}(V_1,\vph_1)\lra \La^{\top}_{\C}(V_2,\vph_2)$$
be the induced element of $\Isom(\La^{\top}_{\C}(V_1,\vph_1),\La^{\top}_{\C}(V_2,\vph_2))$.
Define
\begin{equation*}\begin{split}
\Isom'\big((V_1,\vph_1),(V_2,\vph_2)\big)
=\big\{(f,\psi)\!\in\!
\Isom\big(\La^{\top}_{\C}(V_1,\vph_1),\La^{\top}_{\C}(V_2,\vph_2)\big)\!\times\!
\Isom\big(V_1^{\vph_1},V_2^{\vph_2}\big)\!:\,&\\
f|_{\La_{\R}^{\top}V^{\vph_1}}\!=\!\La_{\R}^{\top}\psi&\big\}.
\end{split}\end{equation*}
The next statement is an immediate consequence of Theorems~\ref{nodalBHH_thm} 
and~\ref{auto_thm}.

\begin{crl}[of Theorems~\ref{nodalBHH_thm},\ref{auto_thm}]\label{auto_crl}
Let $(\Si,\si)$ be a (possibly nodal) symmetric surface and 
$(V_1,\vph_1)$ and~$(V_2,\vph_2)$ over~$(X,\phi)$ such that 
$$\rk_{\C}V_1=\rk_{\C}V_2, \qquad
w_1\big(V_1^{\vph_1}\big)=w_1\big(V_2^{\vph_2}\big)\in H^1(\Si^{\si};\Z_2),$$
and $\deg(V_1|_{\Si'})\!=\!\deg(V_2|_{\Si'})$ for each irreducible component $\Si'\!\subset\!\Si$.
Then the maps
\begin{alignat*}{2}
\Isom\big((V_1,\vph_1),(V_2,\vph_2)\big)&\lra 
\Isom'\big((V_1,\vph_1),(V_2,\vph_2)\big), &\quad
&\Psi\lra\big(\La_{\C}^{\top}\Psi,\Psi|_{V_1^{\vph_1}}\big),\\
\Isom'\big((V_1,\vph_1),(V_2,\vph_2)\big)
&\lra \Isom\big(\La^{\top}_{\C}(V_1,\vph_1),\La^{\top}_{\C}(V_2,\vph_2)\big),
&\quad&(f,\psi)\lra f,\\
\Isom'\big((V_1,\vph_1),(V_2,\vph_2)\big)
&\lra \Isom\big(V_1^{\vph_1},V_2^{\vph_2}\big), 
&\quad&(f,\psi)\lra\psi,
\end{alignat*}
are surjective.
Two isomorphisms from $(V_1,\vph_1)$ to $(V_2,\vph_2)$
lie in the same path component of $\Isom((V_1,\vph_1),(V_2,\vph_2))$
if and only if their images in
$$\Isom\big(\La^{\top}_{\C}(V_1,\vph_1),\La^{\top}_{\C}(V_2,\vph_2)\big) 
\qquad\hbox{and}\qquad \Isom\big(V_1^{\vph_1},V_2^{\vph_2}\big)$$
lie in the same path components of the two spaces.
Furthermore, every path $(f_t,\psi_t)$ in the space $\Isom'((V_1,\vph_1),(V_2,\vph_2))$
passing through the images of some 
\hbox{$\Psi,\Phi\!\in\!\Isom((V_1,\vph_1),(V_2,\vph_2))$} lifts to a path 
in $\Isom((V_1,\vph_1),(V_2,\vph_2))$ passing through~$\Psi$ and~$\Phi$. 
\end{crl}

\noindent
A special case of this corollary underpins the perspective on orientability
in real Gromov-Witten theory and for orienting naturally twisted determinants
of Fredholm operators on real bundle pairs over symmetric surfaces
introduced in~\cite{RealOrient} and built upon in~\cite{RealGWsI,RealGWsII}.
This perspective motivated the following definition.

\begin{dfn}[{\cite[Definition~5.1]{RealGWsI}}]\label{realorient_dfn4}
Let $(X,\phi)$ be a topological space with an involution and 
$(V,\vph)$ be a real bundle pair over~$(X,\phi)$.
A \sf{real orientation} on~$(V,\vph)$ consists~of 
\begin{enumerate}[label=(RO\arabic*),leftmargin=*]

\item\label{LBP_it2} a rank~1 real bundle pair $(L,\wt\phi)$ over $(X,\phi)$ such that 
\BE{realorient_e4}
w_2(V^{\vph})=w_1(L^{\wt\phi})^2 \qquad\hbox{and}\qquad
\La_{\C}^{\top}(V,\vph)\approx(L,\wt\phi)^{\otimes 2},\EE

\item\label{isom_it2} a homotopy class of isomorphisms of real bundle pairs in~\eref{realorient_e4}, and

\item\label{spin_it2} a spin structure~on the real vector bundle
$V^{\vph}\!\oplus\!2(L^*)^{\wt\phi^*}$ over~$X^{\phi}$
compatible with the orientation induced by~\ref{isom_it2}.\\ 

\end{enumerate}
\end{dfn}

\noindent
We recall that a \sf{spin structure} on a rank~$n$ oriented real vector bundle 
$W\!\lra\!X^{\phi}$ with a Riemannian metric is a $\Spin_n$ principal bundle
$\wt\Fr(W)\!\lra\!X^{\phi}$ factoring through a double cover $\wt\Fr(W)\!\lra\!\Fr(W)$
equivariant with respect to the canonical homomorphism $\Spin_n\!\lra\!\SO_n$.
Since $\SO_n$ is a deformation retract of the identity component $\GL_n^+\R$ of~$\GL_n\R$,
this notion is independent of the choice of the metric on~$W$.
If $X^{\phi}$ is a CW-complex, a spin structure on~$W$ is equivalent to 
a trivialization of~$W$ over the 2-skeleton of~$X^{\phi}$.\\

\noindent
An isomorphism~$\Th$ in~\eref{realorient_e4} restricts to an isomorphism 
\BE{realorient2_e3}\La_{\R}^{\top}V^{\vph}\approx (L^{\wt\phi})^{\otimes2}\EE
of real line bundles over~$X^{\phi}$.
Since the vector bundles $(L^{\wt\phi})^{\otimes2}$ and $2(L^*)^{\wt\phi^*}$ are canonically oriented, 
$\Th$~determines orientations on $V^{\vph}$ and $V^{\vph}\!\oplus\! 2(L^*)^{\wt\phi^*}$.
We will call them \textsf{the orientations determined by~\ref{isom_it2}}
if $\Th$ lies in the chosen homotopy class.
An isomorphism~$\Th$ in~\eref{realorient_e4} also induces an isomorphism 
\BE{detInd_e}\begin{split}
\La_{\C}^{\top}\big(V\!\oplus\!2L^*,\vph\!\oplus\!2\wt\phi^*\big)
&\approx \La_{\C}^{\top}(V,\vph)\otimes(L^*,\wt\phi^*)^{\otimes 2}\\
&\approx (L,\wt\phi)^{\otimes 2}\otimes(L^*,\wt\phi^*)^{\otimes 2}
\approx \big(\Si\!\times\!\C,\si\!\times\!\fc\big),
\end{split}\EE
where the last isomorphism is the canonical pairing.
We will call the homotopy class of isomorphisms~\eref{detInd_e} induced by 
the isomorphisms~$\Th$ in~\ref{isom_it2} \textsf{the homotopy class determined by~\ref{isom_it2}}.\\

\noindent
By the above,
a real orientation on a rank~$n$ real bundle pair~$(V,\vph)$ over
a symmetric surface  $(\Si,\si)$ determines a topological component
of the~space
\begin{equation*}\begin{split}
&\Isom'\big((V\!\oplus\!2L^*,\vph\!\oplus\!2\wt\phi^*),
(\Si\!\times\!\C^{n+2},\si\!\times\!\fc)\big)\\
&\quad\subset
\Isom'\big(\La_{\C}^{\top}(V\!\oplus\!2L^*,\vph\!\oplus\!2\wt\phi^*),
\La_{\C}^{\top}(\Si\!\times\!\C^{n+2},\si\!\times\!\fc)\big)
\times
\Isom'\big((V\!\oplus\!2L^*)^{\vph\oplus2\wt\phi^*},
(\Si\!\times\!\C^{n+2})^{\si\times\fc}\big).
\end{split}\end{equation*}
The next proposition, established for smooth and one-nodal symmetric surfaces 
in~\cite{RealGWsI} and for symmetric surfaces with one pair of conjugate nodes
in~\cite{RealGWsII},  is thus a special case of Corollary~\ref{auto_crl}.

\begin{prp}\label{canonisom_prp}
Let $(\Si,\si)$ be a symmetric surface and
$(V,\vph)$ be a rank~$n$ real bundle pair over $(\Si,\si)$.
A real orientation on $(V,\vph)$
determines a homotopy class of isomorphisms
$$ \Psi\!: \big(V\!\oplus\!2L^*,\vph\!\oplus\!2\wt\phi^*\big)
\approx\big(\Si\!\times\!\C^{n+2},\si\!\times\!\fc\big)$$
of real bundle pairs over $(\Si,\si)$.
An isomorphism~$\Psi$ belongs to this homotopy class if and only~if
the restriction of $\Psi$ to the real locus induces the chosen spin structure~\ref{spin_it2} 
and  the isomorphism 
$$\La_{\C}^{\top}\Psi\!: \La_{\C}^{\top}\big(V\!\oplus\!2L^*,\vph\!\oplus\!2\wt\phi^*\big)
\lra \La_{\C}^{\top}\big(\Si\!\times\!\C^{n+2},\si\!\times\!\fc\big)
=\big(\Si\!\times\!\C,\si\!\times\!\fc\big)$$
lies in the homotopy class determined by~\ref{isom_it2}.
\end{prp}
 
\noindent
A \sf{symmetric Riemann surface} $(\Si,\si,\fj)$ is a symmetric surface with
a complex structure~$\fj$ on~$\Si$ such that $\si^*\fj\!=\!-\fj$.
A \textsf{real Cauchy-Riemann operator} on a real bundle pair $(V,\vph)$ 
over such a surface is a linear map of the~form
\begin{equation*}\begin{split}
D=\dbar_V\!+\!A\!: \Ga(\Si;V)^{\vph}
\equiv&\big\{\xi\!\in\!\Ga(\Si;V)\!:\,\xi\!\circ\!\si\!=\!\vph\!\circ\!\xi\big\}\\
&\hspace{.1in}\lra
\Ga_{\fj}^{0,1}(\Si;V)^{\vph}\equiv
\big\{\ze\!\in\!\Ga(\Si;(T^*\Si,\fj)^{0,1}\!\otimes_{\C}\!V)\!:\,
\ze\!\circ\!\tnd\si=\vph\!\circ\!\ze\big\},
\end{split}\end{equation*}
where $\dbar_V$ is the holomorphic $\dbar$-operator for some holomorphic structure in~$V$ and  
$$A\in\Ga\big(\Si;\Hom_{\R}(V,(T^*\Si,\fj)^{0,1}\!\otimes_{\C}\!V) \big)^{\vph}$$ 
is a zeroth-order deformation term. 
Let $\dbar_{\Si;\C}$ denote the real Cauchy-Riemann operator on 
the trivial rank~1 real bundle
\hbox{$(\Si\!\times\!\C,\si\!\times\fc)$}
with the standard holomorphic structure and $A\!=\!0$.
Any real Cauchy-Riemann operator~$D$ on a real bundle pair
 is Fredholm in the appropriate completions.
We denote~by
$$\det D\equiv\La_{\R}^{\top}(\ker D) \otimes \big(\La^{\top}_{\R}(\text{cok}\,D)\big)^*$$
its \textsf{determinant line}.\\

\noindent
If $(X,\phi)$ is a topological space with an involution,
a \sf{real map} $u\!:(\Si,\si)\!\lra\!(X,\phi)$ is a continuous map $u\!:\Si\!\lra\!X$
such that $u\!\circ\!\si\!=\!\phi\!\circ\!u$.
Such a map pulls back a real bundle pair $(V,\vph)$ over $(X,\phi)$
to a real bundle $u^*(V,\vph)$ over $(\Si,\si)$
and a real orientation on the former to a real orientation on the latter.
By Proposition~\ref{canonisom_prp}, a real orientation on a rank~$n$ real bundle 
pair~$(V,\vph)$ over $(X,\phi)$ thus determines an orientation on 
the \sf{relative determinant}
\BE{fDdfn_e}
\rdet\,D_u\equiv \big(\!\det D_u\big)\otimes\big(\!\det\dbar_{\Si;\C}\big)^{\otimes n}\EE
for every real Cauchy-Riemann operator~$D_u$ 
on the real bundle pair $u^*(V,\vph)$ over $(\Si,\si)$
for every real map $u\!:(\Si,\si)\!\lra\!(X,\phi)$.
This observation plays a central role in the construction of positive-genus real Gromov-Witten
invariants in~\cite{RealGWsI}.
If $\Si$ is of genus~0, $\det\dbar_{\Si;\C}$ has a canonical orientation and 
an orientation on~$\rdet\,D_u$ is canonically equivalent to an orientation \hbox{on~$\det D_u$}.
In particular, \cite[Theorem~1.3]{Teh} is essentially equivalent to 
the case of this observation
with $\Si\!=\!\P^1$ and $\si$ being an involution without fixed points.\\

\noindent
The analogue of the last factor in~\eref{fDdfn_e} in the case of bordered surfaces
is canonically oriented.
Thus, the role of the  real orientations of Definition~\ref{realorient_dfn4}
for Cauchy-Riemann operators in real Gromov-Witten theory is analogous
to that of the relative spin structures of \cite[Definition~8.1.2]{FOOO} 
in open Gromov-Witten theory.\\

\vspace{.5in}

\noindent
{\it  Institut de Math\'ematiques de Jussieu - Paris Rive Gauche,
Universit\'e Pierre et Marie Curie, 
4~Place Jussieu,
75252 Paris Cedex 5,
France\\
penka.georgieva@imj-prg.fr}\\

\noindent
{\it Department of Mathematics, Stony Brook University, Stony Brook, NY 11794\\
azinger@math.stonybrook.edu}\\

\end{document}